\documentclass[english,11pt]{amsart}
\usepackage{amssymb,amsmath,amsthm}
\usepackage{srcltx}
\usepackage{babel,graphicx}
\usepackage{color}
\newtheorem{theorem}{Theorem}[section]

\newtheorem{corollary}[theorem]{Corollary}
\newtheorem{lemma}[theorem]{Lemma}
\newtheorem{proposition}[theorem]{Proposition}

\theoremstyle{definition}
\newtheorem{definition}[theorem]{Definition}

\newtheorem{question}[theorem]{Question}

\theoremstyle{remark}
\newtheorem{remark}[theorem]{Remark}

\numberwithin{equation}{section}

\newcommand{\fin}{\mbox{}\hfill $\square$ \\[0.2cm]}

\newcommand{\norm}[1]{\left\Vert#1\right\Vert}

\newcommand{\R}{\mathbf{R}}
\newcommand{\C}{\mathbf{C}}
\newcommand{\D}{\mathbf{D}}

\newcommand{\FF}{\mathcal{F}}

\newcommand{\oo}{\gamma}
\newcommand{\OO}{\Gamma}

\newcommand{\Prob}{{\mathbb P}}

\makeatletter
\@namedef{subjclassname@1991}{2020 Mathematics Subject Classification}
\makeatother

\begin{document}

\title[Convexity of complements of limit sets]{Convexity of complements of limit sets for holomorphic foliations on surfaces}
\author{Bertrand Deroin, Christophe Dupont \& Victor Kleptsyn}

\address{CNRS-Laboratoire AGM-Universit\'e de Cergy-Pontoise}
\email{bertrand.deroin@cyu.fr}

\address{IRMAR-Universit\'e de Rennes 1}
\email{christophe.dupont@univ-rennes1.fr}

\address{CNRS-IRMAR-Universit\'e de Rennes 1}
\email{victor.kleptsyn@univ-rennes1.fr}






\begin{abstract} 
Let $\FF$ be a  holomorphic foliation on a compact K\"ahler surface with hyperbolic singularities and no foliation cycle. We prove that if the limit set of $\FF$ has zero Lebesgue measure, then its complement is a modification of a Stein domain. The proof consists in building, in several steps, a metric of positive curvature for the normal bundle of $\FF$ near the limit set. Then we construct a proper strictly plurisubharmonic exhaustion function for the complement of the limit set, by adapting Brunella's ideas to our singular context. The arguments hold more generally when the limit set is thin, a property relying on Brownian motion. 
\end{abstract}

\maketitle

\section{Introduction} 

\subsection{Limit sets, minimal sets, and the geometry of their complements} A \emph{limit set} $\mathcal L$ of a holomorphic foliation $\FF$ on a complex surface is a closed saturated subset contained in the closure of every leaf of $\FF$, it is unique if it exists. Limit sets  have in general a fractal structure. Classical examples arise with Riccati foliations, namely foliations on surfaces transverse to a rational fibration. In this context, the limit set corresponds to the limit set of the monodromy group, which can be any Kleinian group. Examples of strict limit sets in the non linear context have been discovered recently, e.g. Jouanolou foliation in degree two of $\mathbb P^2$ and its perturbations, see Alvarez-Deroin \cite{Alvarez}.

More generally, a \emph{minimal set} $\mathcal M$ is a closed saturated subset such that every leaf of $\FF$ contained in $\mathcal M$ is dense in $\mathcal M$.   
Complements of minimal sets display  very interesting geometrical properties. For instance, Grauert's classical example of pseudoconvex  but not Stein domain arises as complement of minimal sets of irrational linear foliations on complex tori \cite{Pet}. Recall that a domain of a complex manifold is \emph{Stein} if it is biholomorphic to a domain of a complex affine space \cite{H}. Remarkably, this is characterized by the existence of a proper and strictly plurisubharmonic exhaustion function on the domain \cite{Grauert}. Such exhaustions can be used to perform Hartogs fillings, which allow to extend analytic objects. For instance, pseudoconvex subsets of $\mathbb P^n$ being  Stein \cite{Takeuchi}, one can  extend to $\mathbb P^n$ the Cauchy-Riemann foliation of  analytic Levi-flat hypersurfaces. Lins Neto enforced that powerful idea to show that there is no Levi-flat hypersurface in $\mathbb P^n$ for $n \geq 3$ \cite{LinsNeto}. That tricky question remains open   in $\mathbb P^2$.   

In an other nice work, Brunella proved that the complement of a compact saturated set avoiding the singular set of $\FF$ is a modification of a Stein domain  once the normal bundle $N_\FF$  has positive curvature near the  compact set  \cite{BrunellaToulouse} (see \cite{Pet} for the notion of modification). In the context of Levi-flat hypersurfaces, Canales  reached positivity for $N_\FF$  by using dynamical properties on the Cauchy-Riemann foliation, and managed to adapt Brunella's arguments \cite{Canales}. In particular, she retrieved, by a different way, Diederich-Ohsawa's convexity of complements of Levi-flat hypersurfaces which are limit sets of Riccati foliations with real monodromy \cite{DO1,DO2}.

\subsection{Main result} The novelties of our study is that we deal with limit sets that contain singular points of the foliation and which may have a fractal structure. We work with foliations on a compact K\"ahler surface satisfying
\vspace{0.2cm}
 
\textit{(*) every singular point of \(\FF\) is hyperbolic (the eigenvalues of the vector field are not \(\mathbb R\)-colinear) 
and \( \FF \) does not carry any foliation cycle. }

\vspace{0.2cm}

This condition is generic in many algebraic families of foliations, for instance it is satisfied on a real Zariski open dense subset in the moduli space of degree \(d\) holomorphic foliations of \(\mathbb P^2\). Moreover, under the condition (*), Dinh-Nguyen-Sibony \cite{DNS_unique ergodicity} proved that there exists on the surface a unique $\FF$-directed harmonic current. That implies that $\FF$ has a limit set $\mathcal L$, given by the support of the harmonic current. We obtain the following result. 

\begin{theorem} \label{t: convexity II}
Let \(\FF\) be a holomorphic foliation of a compact K\"ahler surface  satisfying property (*). If the limit set $\mathcal L$ is thin (in particular if it has zero Lebesgue measure), then its complement is a modification of a Stein domain. 
\end{theorem} 

The proof consists in showing that the normal bundle $N_\FF$ supports a metric with positive curvature in a neighborhood of $\mathcal L$, we proceed in 3 steps explained below. We then complete the proof by constructing a proper strictly plurisubharmonic exhaustion function near $\mathcal L$, that fourth step needs to adapt Brunella's arguments \cite{BrunellaToulouse} to our singular context, the fact that the singular points  are linearizable is crucial.

In the \emph{first step} we prove that the curvature is positive along the foliation: we follow the approach of \cite{Deroin-Kleptsyn}, using the absence of foliation cycle and Hahn-Banach's theorem. It is interesting to notice that we get a quick proof of the negativity of the Lyapunov exponent, established in \cite{Nguyen3}.
The \emph{second step} establishes positivy of the curvature in neighborhoods of singular points. Here we use suitable notions of positivy for currents,  the vanishing of the Lelong numbers of the harmonic current \cite{Nguyen2} and the uniqueness of the harmonic current \cite{DNS_unique ergodicity}. 

The \emph{third step} is based on the thin property, opened to reach positivity for $N_\FF$ in a neighborhood of $\mathcal L$ outside the singular points. This is a potential theoritic condition on the fractal structure of the limit set. After Doob, a closed subset \(K\subset {\bf D}\)  is thin if for every point \(p\in K\), the probability that a Brownian trajectory starting at \(p\) stays in \(K\) during a positive amount of time is vanishing. The limit set $\mathcal L$ is thin if its image under local first integrals with values in the unit disc is thin.  We shall use that thin sets have neighborhoods whose first eigenvalue with respect to the Dirichlet problem can be made arbitrarily large, this property actually characterizes them.
  
 \subsection{Kleinian groups and Julia sets} For Riccati foliations, the limit set is the whole surface or has zero Lebesgue measure. This is a consequence of the solution of Ahlfors' conjecture for Kleinian groups, see the combination of works due to Ahlfors \cite{Ahlfors},  Agol \cite{Agol}, Calegari-Gabai  \cite{Calegari_Gabai}  and Canary \cite{Canary}. Since a limit set with zero Lebesgue measure is thin (Lemma \ref{lemma: positiveLeb}), Theorem \ref{t: convexity II} implies:

\begin{corollary}\label{c: Riccati}
The complement of the limit set of a Riccati foliation is a modification of a Stein domain. 
\end{corollary}

This result extends the works of Diederich-Ohsawa and Canales mentionned above, which concern Levi-flat hypersurfaces. We do not  known any example of foliation having a limit set which is not the whole ambiant surface but has positive Lebesgue measure. Observe that this is not the case for the Julia set of polynomial mappings acting on $\C$, see Buff-Ch\'eritat \cite{BC} and Avila-Lyubich \cite{AL}. This motivates the following general questions.
\begin{question}
Is it true that the limit set of a holomorphic foliation is either the whole ambiant space or is thin? Is it true that Julia sets of rational mappings are thin?
\end{question}
These questions are out of reach for the present work, but to motivate them, we prove that

\begin{theorem}\label{t: Julia polynomial} The Julia set of a polynomial mapping acting on $\C$ is thin.\end{theorem}  

\subsection{Organization of the article} In Section \ref{s: preliminaries}, we present general facts on holomorphic foliations and harmonic currents, and explain with Proposition \ref{l: Hahn-Banach} how functional analysis enters the picture. 
In the next sections, as specified above, we show in 3 steps that the normal bundle $N_\FF$ supports a metric with positive curvature in a neighborhood of $\mathcal L$. 
In Section~\ref{s: positivity normal bundle}, Theorem \ref{t: positivity normal bundle} and its improvement Theorem \ref{c: positivity II} corresponds to the first two steps. Thin sets are introduced in Section \ref{s: thin sets}, then one can proceed to the third step in Section \ref{s: positivity all directions}. The delicate construction of the proper and strictly plurisubharmonic exhaustion function near $\mathcal L$ occupies Section \ref{s: convexity}, the non-trivial monodromy near the singular points leads us to introduce our so-called $m$-functions.  The last section is dedicated to the proof of Theorem \ref{t: Julia polynomial}. \\

\emph{Acknowledgements.} We thank Alano Ancona for pointing the reference \cite{BG} on the thin property and Misha Lyubich for discussions about that property for Julia sets. We also thank the Research in Paris program of Institut Henri Poincar\'e for the very nice working conditions offered to us during fall 2021.

\section{Preliminaries} \label{s: preliminaries}

\subsection{Tangent bundle, leaves, first integrals} 

A holomorphic foliation \(\FF\) on a smooth K\"ahler surface \(S\) is the data of a line bundle \( T\FF \rightarrow S\) and a morphism \(\pi : T\FF \rightarrow TS\) which vanishes only at a finite number of points. Such a point is called a \emph{singularity} of \( \FF\), and their set is denoted by \(\text{sing}(\FF)\). Let \(S^*:= S\setminus \text{sing}(\FF) \) be the regular part of \(\FF\). The bundle \(T\FF\) is called the \emph{tangent bundle} of $\FF$. The vector fields on \(S\) that are images of local sections of \(T\FF\) by \( \pi \) form a subsheaf of \(\mathcal O (TS)\) called the \emph{tangent sheaf} of \(\FF\). 

The \emph{leaves} of the foliation are the equivalence classes of the relation on $S^*$ defined by: two points are equivalent if they belong to the same integral curve of a locally defined vector field belonging to the tangent sheaf of \(\FF\). A local \emph{first integral} of \(\FF\) is a function \( t : W \rightarrow {\bf C}\) defined on an open subset of \(S\), which is constant along the leaves of the restriction of $\FF$ to \(W\).  

On a neighborhood \(W\) of any regular point \(p\) is defined a biholomorphism \( (z, t) : W \rightarrow {\bf D}\times {\bf D}\) to the bidisc that maps the tangent sheaf of \(\FF\) to the sheaf of horizontal vector field of the bidisc, and which is called a \emph{foliation chart}. The function \( t: W \rightarrow {\bf D}\) is then a local first integral of the foliation and it is a submersion. Any other first integral in \(W\) is a function of \(t\). In a foliation chart, a set of the form \( {\bf D}\times \{t\}\) is called a \emph{plaque}.

\subsection{Normal bundle}
It is defined on $S^*$ by \( N_{\FF} = TS / \pi (T\FF) \). Any metric \( m \) on \( N_{\FF} \) has the following form 
\begin{equation}\label{eq: varphi}  m = \exp (2\varphi_m)\ | dt| ^2 ,\end{equation}  
where \( t \) is a local submersion defining $\FF$. On every plaque, the function \( \varphi_m\) is well-defined up to an additive constant, hence any derivative of \(\varphi_m\) is well-defined (in the sequel, we will consider the gradient and Laplacian of \(\varphi_m\) wrt a metric on \( T\FF\)).  We introduce on $S^*$ the leafwise \(1\)-form 
\begin{equation}\label{eq: eta} \eta_m := d_{\FF} \varphi_m .
\end{equation}
We will use that if \( m =  \exp (\psi) m' \) is another smooth metric on \( N_{\FF}\), then  
\begin{equation} \label{eq: change of metric} \eta_m = d_{\FF} \psi + \eta_{m'} .  \end{equation} 
If \(\omega \) is a local non vanishing section of the dual line bundle \( N_{\FF}^*\), that we think as a holomorphic one form on \(S\) that vanishes on \( T\FF\), the modulus \(|\omega|\) defines a hermitian metric on \( N_{\FF}\). Writing in local charts \( \omega = f(z,t) dt \) for a holomorphic function \( f\), we have the local expression 
$$ \eta _{|\omega|} = d_{\FF} \log |f(z,t)| .$$
 On the other hand, the form \( \alpha_{\omega} = d_{\FF} \log f \) defines a local section of the canonical bundle \( K_{\FF} := T^* \FF\), which depends only on \(\omega\), and which satisfies the equation \begin{equation} \label{eq: alpha} d\omega = \alpha_{\omega} \wedge \omega . \end{equation}
With this form at hand, we have the formula 
\begin{equation} \label{eq: eta omega} \eta_{|\omega|} = \Re \alpha_{\omega}.\end{equation}

\subsection{Singular points}\label{sub:singpoints} At the neighborhood of a singular point \( p \), there is a germ of vector field \(V\) belonging to the tangent sheaf of the foliation, that vanishes only at \( p \). By Hartog's lemma, such a vector field is the \( \pi \)-image of a generating section of \(T\FF\). Any other germ of vector field with these properties differ from \(V\) by multiplication by a non vanishing holomorphic function. We will say that the singularity \( p \) is \emph{hyperbolic} if the two eigenvalues of \(V\) are not \({\bf R}\)-colinear. 

By Poincar\'e linearization theorem, see e.g. \cite[Theorem 5.5, p. 50]{IY}, there exist  coordinates \((x,y)\) onto the bidisc \( \D\times \D \) on which 
\begin{equation}\label{eq: linearization} V = a x \partial _x + b y \partial _y .\end{equation} 
Such coordinates \((x,y)\) are called \emph{linearization coordinates}. 
 
As for the tangent bundle, the normal bundle extends in a unique way to a line bundle on $S$. It is sufficient to verify this for the conormal bundle: in linearization coordinates around a singular point $p$, the holomorphic \(1\)-form   
\[\omega =  a x dy - b y dx  \]
vanishes exactly on \(T\FF\) in \(S^*\), hence defines a section of \( N_{\FF} ^* \) in \(S^*\)  that does not vanish. This section extends as a section of \(N_{\FF}^*\) defined at $p$ and does not vanish there as well, as claimed. A smooth metric \( m  \) on \( N_{\FF} \) has the following expression close to a singularity 
\begin{equation}\label{eq: smooth metric on NF} m := \exp (2 \psi) \ |\omega|^2 \end{equation}
where \(\psi\) is a smooth function (including at the singularity). We then have 
\[ \eta_{m} = d_{\FF} \psi + \Re \alpha_{\omega} \]
where \( \alpha_{\omega}\) is a section of \( K_{\FF}\), given in linearization coordinates by
\begin{equation}\label{eq: alpha} \alpha_\omega=\frac{a+b}{a} \left(\frac{dx}{x}\right)_{|\FF} = \frac{a+b}{b} \left(\frac{dy}{y}\right)_{|\FF} . \end{equation}

\subsection{Harmonic currents}\label{ss: harmonic current}  

 A current of bidimension $(1,1)$ on $S$ is a continuous linear form on the space of smooth \((1,1)\)-forms  \(\Omega^{1,1}(S)\).  In this article, every current will be of bidimension $(1,1)$. 
 
 \begin{definition}
 Let \( \mathcal P \) be a closed subset of  \(TS \). A $(1,1)$-form $\omega$ is \(\mathcal P\)-positive if \( \omega (u, iu ) \geq 0\) for every \( u\in \mathcal P\) (written $\omega_{\vert \mathcal P} \geq 0$). A current $T$ is \(\mathcal P\)-positive if \(T (\omega) \geq 0\) for every \(\mathcal P\)-positive $(1,1)$-form $\omega$. \end{definition}
 
  If \(\mathcal P \subset \mathcal P'\), then any \(\mathcal P\)-positive current is  \(\mathcal P'\)-positive. In particular every $\mathcal P$-positive current is positive in the usual sense.
  
 \begin{lemma}\label{lem: largestrict}
Let $T$ be a non trivial current. Then $T$ is \(\mathcal P\)-positive if and only if $T(\omega) > 0$ as soon as $\omega_{\vert \mathcal P} > 0$.
 \end{lemma}

\begin{proof}
Let $\kappa$ be a K\"ahler form on $S$ and let $\omega$ be such that $\omega_{\vert \mathcal P} > 0$. Since $\mathcal P$ is closed, there exists $\epsilon >0$ such that $(\omega - \epsilon \kappa)_{\vert \mathcal P} \geq 0$. If $T$ is  non trivial and \(\mathcal P\)-positive, we get $T(\omega) \geq \epsilon T(\kappa) > 0$ as desired. Reciprocally, consider  $\omega + \epsilon \kappa$ and let $\epsilon$ tend to zero.
\end{proof}
 
A current \(T\) is \textit{harmonic} if it is \(dd^c\)-closed, namely \(T ( d d^c f ) = 0 \) for every smooth function \(f\in C^\infty (S)\). Skoda \cite{Skoda} proved that if \( T\) is a positive harmonic current defined at the neighborhood of the origin in \({\bf C}^2\), then 
\begin{equation} \label{eq: monotonicity trace measure} r\mapsto I_T(r):=\frac{1}{r^2}  \int_{|x|^2 + |y|^2 \leq r^2}  T \wedge i (dx \wedge d\overline{x} + dy \wedge d\overline{y} ) \end{equation}
is non decreasing. The limit when \(r\) tends to \(0\) therefore exists, it is the Lelong number  of \(T\) at the origin, it does not depend on the coordinate system. In particular, a harmonic current does not put any mass on points.

A harmonic current \( T\) defines an element \( [T]\) in the dual \( H^{1,1}_{BC} (S, \R) ^* \) of the Bott-Chern cohomology group, which is isomorphic to \( H^{1,1} (S, \R) ^*\) by the \(dd^c\)-lemma. By duality, it defines a class \([T]\) in \( H^{1,1} (S, \R) \). Its intersection with the Chern class \(c_1(L)\) of a line bundle \( L\) is defined by 
\begin{equation} \label{eq: intersection with harmonic current }  [T] \cdot c_1(L) = T ( \Theta_{m}  ) \end{equation}  
where \( m \) is any hermitian metric on \(L\) and \(\Theta_{m}  = - \frac{1}{2} d d^c \log m(s) \) is the curvature form of \(m\), \(s\) being a local non vanishing holomorphic section of $L$, see \cite{Ghys}. We shall denote \( T\cdot L\) this intersection.  The following application of Hahn-Banach separation principle will be crucial. 

\begin{proposition} \label{l: Hahn-Banach} 
Let \(\mathcal P\subset TS\) be a closed subset. A line bundle \(L\) has a hermitian metric $m$ whose curvature satisfies $(\Theta_m)_{\vert \mathcal P} >0$  if and only if \( T\cdot L >0\) for every non trivial \(\mathcal P\)-positive harmonic current \(T\). 
\end{proposition} 

\begin{proof}
If $m$ is a hermitian metric on \(L\) such that $(\Theta_m)_{\vert \mathcal P} >0$, then Lemma \ref{lem: largestrict} gives \( T \cdot L >0\) for every non trivial \(\mathcal P\)-positive harmonic current \(T\). 

Reciprocally, assume that \( T \cdot L >0\) for every non trivial \(\mathcal P\)-positive harmonic current \(T\). Let \( E \) be the Banach space of \( (1,1)\)-forms on \(S\) with continuous coefficients, equipped with the topology of uniform convergence,  \( F\subset E\) be the subspace of \( dd^c \)-exact smooth \((1,1)\)-forms, \(\mathcal C\subset E\) be the convex open cone formed by continuous \((1,1)\)-forms $\omega$ such that $\omega_{\vert \mathcal P} > 0$, and \(m' \) be a hermitian metric on \(L\). If \( \Theta_{m'} + F\) does not intersect \(\mathcal C\), Hahn-Banach separation theorem asserts that there exists a linear functional \( T : E\rightarrow {\bf R}\) and \(s\in {\bf R}\) such that for any \(  x\in \Theta_{m'} +F\) and any \(y\in \mathcal C\), \( T(y) > s \geq T(x)\). Since \(\mathcal C\) is a cone, and \(T\) is bounded from below by $s$ on \(\mathcal C\), the infimum of \(T\) on \(\mathcal C\) is \(0\), so we can assume that \(s=0\). Note also that \(T\) is bounded from above on \(F\), and since \(F\) is a linear subspace, it must vanish identically on \(F\). So \(T\) is a non trivial \(\mathcal P\)-positive harmonic current and \( T \cdot L = T(\Theta_{m'}) \leq 0\), contradiction.  Hence \( \Theta_{m'} + F\) intersects \(\mathcal C\); since this is the set of curvatures of smooth hermitian metrics on \(L\), there exists a hermitian metric $m$ on \(L\) whose curvature belongs to $\mathcal C$, that is  $(\Theta_m)_{\vert \mathcal P} >0$.
\end{proof}

\begin{remark}\label{rk: imp} We also have that if $T \cdot L \leq 0$ then there exists for every $\varepsilon > 0$ a hermitian metric $m_\varepsilon$ on $L$ such that  $(\Theta_m)_{\vert \mathcal P} \leq \varepsilon \kappa_{\vert \mathcal P}$. The changes consist in replacing  \(\overline{\mathcal C}\) by \(\mathcal C\) and to use the relevant Hahn Banach theorem to get \( T(y) <  0 < T(x)\). The contradiction implies that $-\Theta_{m'} + F$ intersects \(\overline{\mathcal C}\).
\end{remark} 

\subsection{Directed harmonic currents}\label{ss: directed harmonic current} 

Let \(\mathcal P_{\mathcal F} \) be the closure of the image of \( T\mathcal F\) in \(TS\), this is the union of the tangent space of \(\mathcal F\) in the regular part of $\FF$ and of the tangent spaces of \(S\) at the singular points of $\FF$. We say that a current is \emph{directed} if it is \(\mathcal P_{\mathcal F}\)-positive. A closed directed current is called a \emph{foliation cycle}, terminology due to Sullivan \cite{Sullivan}. The following lemma shows that our definition of harmonic directed currents coincides with the one employed by Berndtsson-Sibony \cite{BS} and Dinh-Nguyen-Sibony  \cite{DNS_unique ergodicity}.

\begin{lemma}
Let $T$ be a harmonic current. Then $T$ is directed if and only if $T \wedge \Omega = 0$ for every smooth $1$-form $\Omega$ locally defining $\FF$. 
\end{lemma}

\begin{proof}
It suffices to work on $S^*$ since harmonic currents $T$ do not put any mass on points. Assume that $T$ is directed. Let $U$ be a small open neighborhood of a regular point of $\FF$, let $(z,t)$ be holomorphic coordinates  on $U$ such that $dt$ defines $\FF$. Since $T$ is positive, we can write on $U$:
$$T = \alpha i dz\wedge d\bar z + \beta i dz \wedge d\bar t + \bar \beta i dt \wedge d\bar z + \gamma i dt \wedge d\bar t , $$
 where $\alpha, \gamma$ are positive measures and $\beta$ is a complex measure. The form $\eta = - i dt \wedge d\bar t$ satisfies $\eta_{\vert \mathcal P_{\mathcal F}} = 0$, hence $- \alpha = T \wedge \eta$ is a positive measure, therefore $\alpha =0$. Now let $\vert \beta \vert$ be the variation of $\beta$ and let us write $\beta = h \vert \beta \vert$ where $h$ is a measurable function satisfying $\vert h \vert =1$ on $U$. Then $\eta = h i d\bar t \wedge dz + \bar h i d \bar z \wedge dt$ satisfies $\eta_{\vert \mathcal P_{\mathcal F}} = 0$. Hence $-2 \vert \beta \vert = T \wedge \eta$ is a positive measure, which gives $\beta = 0$. Finally, $T = \gamma i dt \wedge d\bar t$ on $U$, hence $T \wedge dt = 0$ as desired. Reciprocally, if $T \wedge \Omega = 0$ for a $1$-form $\Omega$ defining $\FF$, then $T$  has the form $\int h_c [L_c] d\nu(c)$ where $h_c$ is a non negative harmonic function, $[L_c]$ being the current of integration on $\{ t = c\}$ and $\nu$ being a positive measure, see \cite[Proposition 1.6]{BS}. Hence $T$ is directed. \end{proof}

 The existence of directed harmonic currents on foliated compact complex surfaces has been established by Berndtsson-Sibony \cite{BS}. Nguyen \cite{Nguyen2} proved that for foliations satisfying (*), the Lelong number of every directed harmonic current is everywhere vanishing. Dinh, Nguyen and Sibony \cite{DNS_unique ergodicity} recently established that every compact K\"ahler foliated surface satisfying (*) admits a unique directed harmonic current up to multiplication by a constant. Fornaess and  Sibony \cite{FS} previously proved the same result when $S = \mathbb P^2$. The support of this current is a closed saturated subset \(\mathcal L\)   contained in the closure of every leaf, we call it the \emph{limit set} of $\FF$. 

\section{Positivity of the normal bundle along the foliation} \label{s: positivity normal bundle}

In this section, we first prove that the normal bundle of $\FF$ is positive along the foliation. The proof is an adaptation of \cite[Section 3.1.1]{Deroin-Kleptsyn} to our singular context. It is interesting to notice that this provides a quick proof of the negativity of the Lyapunov exponent, established in \cite{Nguyen3}.  In a second time, we improve Theorem \ref{t: positivity normal bundle} by gaining positivity for the normal bundle near the singular set, the arguments rely on Lemma \ref{l: Hahn-Banach} and on the vanishing of the Lelong numbers of the directed harmonic current.

\begin{theorem} \label{t: positivity normal bundle}
Let $S$ be a foliated compact K\"ahler surface satisfying (*). The unique directed harmonic current satisfies $T \cdot N_\FF > 0$. Hence (by Lemma \ref{l: Hahn-Banach}) the normal bundle $N_\FF$ carries a hermitian metric $m$ whose curvature satisfies  \((\Theta_m)_{\vert \mathcal P_\FF} > 0\). 
\end{theorem}
 
 Before proving Theorem \ref{t: positivity normal bundle}, we establish Lemma \ref{l: IBP} below. Let \( m\) be a smooth metric on \( N_{\FF} \) and let \( v_{m}\) be its associated volume form, considered as a global non negative \( (1,1) \)-form on \(S^*\) whose kernel is the tangent bundle \(T\FF\). In local coordinates (see Equation \eqref{eq: varphi}), we have 
\begin{equation} \label{eq: vm} 
 v_m := \exp (2\varphi_m)  \frac{i}{2} dt\wedge d\overline{t} .
 \end{equation} 

\begin{lemma} \label{l: IBP}
The integral 
\begin{equation} \label{eq: integral} \int_{S^*} \left( d _{\FF} d^c_{\FF} \varphi_m + 2 d_{\FF}\varphi_m \wedge d^c_{\FF} \varphi_m \right) \wedge v_m \end{equation} 
is absolutely convergent, and vanishes.
\end{lemma} 

\begin{proof}  
We first prove that the integral \eqref{eq: integral} is absolutely convergent. The problem occurs of course close to the singular points of \(\FF\). There we have \( m = \exp (2\psi) |\omega|^2\) where \( \psi \) is a smooth function, and hence \(v_m = \exp (2\psi) \frac{i}{2} \omega\wedge \overline{\omega}\). Recall that  \(\alpha_\omega\) is the section of \( K_{\FF}\) given by \eqref{eq: alpha}. We then have  
\[ d_{\FF}d_{\FF}^c \varphi_m = d_{\FF}d_{\FF}^c \psi , \ \   d_{\FF} \varphi _m  = \Re \alpha_{\omega} + d_{\FF} \psi , \text{ and } d_{\FF} ^c \varphi _m  = \frac{1}{2\pi} \Im \alpha_{\omega} + d_{\FF}^c \psi.\]
Using the relation \(d\omega = \alpha_\omega \wedge \omega\), one verifies that the forms $d _{\FF} d^c_{\FF} \varphi_m \wedge v_m$ and $d_{\FF}\varphi_m \wedge d^c_{\FF} \varphi_m  \wedge v_m$ are smooth near the singular set, which provides the absolute convergence of \eqref{eq: integral}.

We now prove that \eqref{eq: integral} vanishes. From \eqref{eq: vm} we have \( dv_m = 2 d_{\FF} \varphi_m  \wedge v_m \), so that
\begin{equation}\label{eq: computation of primitive} 
d \left( d_{\FF}^c \varphi_m \wedge v_m \right)  = \left( d _{\FF}d_{\FF}^c \varphi_m  + 2 d_{\FF} \varphi_m \wedge d^c_{\FF} \varphi_m \right) \wedge v_m.\end{equation}
For any compact domain \( U \subset S^*\) with smooth boundary, Stokes formula together with \eqref{eq: computation of primitive} yields 
\begin{equation}\label{eq: IBP} \int _U  \left( d _{\FF}d_{\FF}^c \varphi_m + 2  d_{\FF} \varphi_m \wedge d^c_{\FF} \varphi_m\right)  \wedge v_m = \int_{\partial U} d_{\FF} ^c\varphi_m \wedge v_m.\end{equation}
For \(r >0\) small, take for $U$ the complement  of euclidean balls of radius \(r\) around the singular points in the linearizing coordinates \((x,y)\), denoted \( U_r\). By \eqref{eq: IBP}, to prove that \eqref{eq: integral} vanishes, it suffices to prove that the integral 
\begin{equation*} \label{eq: boundary term} \int _{\partial U_r}  d_{\FF} ^c\varphi_m \wedge v_m \end{equation*}
tends to zero as \( r\) tends to \(0\). This is a consequence of the fact that the \(3\)-form \( d_{\FF} ^c\varphi_m \wedge v_m\) is smooth. 
\end{proof} 



\begin{proof}[Proof of Theorem \ref{t: positivity normal bundle}]
Let \(T\) be the unique directed harmonic current  and assume that \( T \cdot N_{\FF}\leq 0\). Let  \(\kappa\) be a  K\"ahler form on \( S\). Applying Remark \ref{rk: imp}, there exists a family \( \{m^{\varepsilon}\} _{\varepsilon >0} \) of metrics on \(N_{\FF}\) whose curvature form satisfy  \(\left(\Theta_{m^\varepsilon}\right)_{|\mathcal P_\FF} \leq \varepsilon \kappa_{|\mathcal P_\FF} \). 
We normalize each \(m^{\varepsilon}\) so that 
\begin{equation} \label{eq: normalization} \int \kappa \wedge v^{\varepsilon} = 1,\end{equation}
where \(v^{\varepsilon} \) is the transverse volume form associated to \( m^{\varepsilon} \). Introduce the currents 
\begin{equation} \label{eq: approximate foliation cycles} V^\varepsilon (\omega): = \int \omega \wedge v^\varepsilon \text{ for every } \omega\in \Omega^{1,1} (S). \end{equation} 
By compactness of the set of normalized currents equipped with the weak topology, we can find a sequence \(\varepsilon_n\rightarrow 0\) such that \( V^{\varepsilon_n} \) converges to a normalized current \(V\). We claim that  \( V\) is a foliation cycle. 

Since $dd^c \log m^\varepsilon(s) = - \Theta_{m^\varepsilon}$ (see Section \ref{ss: harmonic current}) and \(\left(\Theta_{m^\varepsilon}\right)_{|\mathcal P_\FF} \leq \varepsilon \kappa_{|\mathcal P_\FF} \), the limit current \( V\) is \(\mathcal P_{\mathcal F}\)-positive. It remains to prove that \( V\) is closed. Let  \( \alpha \) be a \( 1\)-form on \(S\). Since 
\[ d(\alpha \wedge v^{\varepsilon}) = d\alpha \wedge v^{\varepsilon} - \alpha \wedge dv^{\varepsilon}= d\alpha \wedge v^{\varepsilon} - 2 \alpha \wedge d_\FF \varphi^{\varepsilon} \wedge v^{\varepsilon}, \]
Stokes formula gives 
\[V^{\varepsilon} (d\alpha ) = 2 \int \alpha \wedge d_\FF \varphi^{\varepsilon} \wedge v^{\varepsilon} . \]
In particular, we get by Cauchy-Schwarz 
\[ |V^{\varepsilon} (d\alpha)| \leq 2 \left( \int \alpha \wedge \alpha^* \wedge v^\varepsilon \right)^{1/2} \left( \int d_{\FF}\varphi^{\varepsilon} \wedge d^c _{\FF}\varphi^{\varepsilon} \wedge v^\varepsilon \right)^{1/2} \]
where \( \alpha ^* (u)= -\alpha (iu)\).  The restriction of \(\alpha \wedge \alpha^*\) to \(\FF\) is bounded by a constant times the K\"ahler form \(\kappa\), hence 
\[ |V^\varepsilon (d\alpha) | \leq C(\alpha) \left( \int d_{\FF}\varphi^\varepsilon \wedge d^c _{\FF}\varphi^\varepsilon \wedge v^\varepsilon \right)^{1/2} ,\]
where \( C(\alpha)\) does not depend on \(\varepsilon\). 
By Lemma \ref{l: IBP}, we have 
\[ \int_{S^*} \left( d _{\FF} d^c_{\FF} \varphi^{\varepsilon} + 2 d_{\FF}\varphi^{\varepsilon}  \wedge d^c_{\FF} \varphi^{\varepsilon} \right) \wedge v^\varepsilon = 0 , \] 
and since \( - d_{\FF} d^c _{\FF} \varphi^{\varepsilon} \) is the restriction of the curvature of \(m^{\varepsilon} \) to \(\FF\), we get from \eqref{eq: normalization} 
\[ 2 \int _{S^*} d_{\FF}\varphi^{\varepsilon}  \wedge d^c_{\FF} \varphi^{\varepsilon}  \wedge v^\varepsilon =\int_{S^*} \Theta_{m^\varepsilon} \wedge v^{\varepsilon}  \leq \varepsilon \int_{S^*} \kappa \wedge v^\varepsilon = \varepsilon . \]
We infer that \( V^{\varepsilon} (d\alpha) \rightarrow_{\varepsilon\rightarrow 0} 0\), and consequently \( V\) is closed. This contradicts the assumption (*) and ends the proof of Theorem \ref{t: positivity normal bundle}. \end{proof}

We now prove the following technical but fundamental refinement of Theorem \ref{t: positivity normal bundle}, which permits to gain positivity for $N_\FF$ at the neighborhood of the singular set. 

\begin{theorem} (Improvement of Theorem \ref{t: positivity normal bundle})  \label{c: positivity II} 
Let \(M >0\) be a constant. There exists a hermitian metric \(m\) on \( N_{\FF}\) whose curvature satisfies 
 \((\Theta_m)_{\vert \mathcal P_\FF} > 0\) and such that for each singular point \( p\in \text{sing} (\FF) \), there exists linearization coordinates \(  (x_p, y_p) : U_p \rightarrow {\bf B} \) from a neighborhood of \(p\) onto the unit ball, such that the foliation in these coordinates is defined by the vector field \eqref{eq: linearization}, and such that in restriction to each \( U_p\), we have: 
\[ \Theta_m > M (i dx_p \wedge d\overline{x_p} + i dy_p \wedge d\overline{y_p} )  .\]
\end{theorem} 

\begin{proof} Let us work near a singular point $p$ and introduce linearization coordinates \( (x,y) : U \rightarrow {\bf B} \) near $p$, see Section \ref{sub:singpoints}. For each \( r\in (0, 1)  \), let \( U_r := \{ |x|^2+|y|^2< r^2\} \) and let
$$ \mathcal P_{\mathcal F, r}:= \mathcal P_{\mathcal F} \cup  \overline{T U (r)} . $$
We recall that $I_T(r)$ is defined in Equation (\ref{eq: monotonicity trace measure}).

\begin{lemma}\label{l: estimation}
Let \(M>0\) be a constant. There exist arbitrarily small radii \(r>0\) such that for every \( \mathcal P_{\mathcal F, r}\)-positive harmonic current \(T_r\), we have 
\[ T_r \cdot N_{\mathcal F} > 4M I_{T_r}(r). \] 
\end{lemma}

\begin{proof}
Suppose to the contrary that for any sufficiently small \(r>0\) there exists a \(\mathcal P_{\mathcal F, r}\)-positive harmonic current \(T_r\) such that \(T_r \cdot N_{\mathcal F} \leq 4M I_{T_r}(r)\). We can assume that \( T_ r\) has total mass \(1\) by normalizing (namely \(T_r (\kappa) = 1\) for a fixed K\"ahler form \(\kappa\) on \(S\)). Since the set of positive currents of mass one is compact for weak convergence, there exists a sequence of radii \(r_n >0\) that tends to \(0\) such that \( T_{r_n}\) weakly converges to a current \(T\). Notice that \(T\) is harmonic, of mass one, and that \( T (\omega) \geq 0\) for every $\omega$ satisfying \(\omega_{\vert \mathcal P_\FF} \geq 0 \). Hence $T$ is directed, in particular \( T \cdot N_{\mathcal F} >0\) by Theorem \ref{t: positivity normal bundle}. 

Let \(\varepsilon >0\) be fixed, and let \(r<\varepsilon\). Since \(I_{T_r}\) is non decreasing, we get  \( I_{T_r} (r) \leq I_{T_r} (\varepsilon) \) so that \(I_{T_r} (\varepsilon) \geq \frac{T \cdot N_{\mathcal F}}{4M}>0\) for any \(0<r < \varepsilon\). Applying this to \(r= r_n\) and letting \(n\) go to \(+\infty\) yields \( I_T (\varepsilon ) \geq \frac{T \cdot N_{\mathcal F}}{4M} \). Being true for every \(\varepsilon>0\), we get that \( \nu ( T, p) \geq \frac{T \cdot N_{\mathcal F}}{4M} \), which contradicts the vanishing of the Lelong number of \(T\) at \(p\).
\end{proof}

Let $M >0$ and $r$ be a small radius provided by Lemma \ref{l: estimation}. Let 
$$ U_p:= U_{r/2} \textrm{ and } (x_p, y_p) = (2x/ r, 2y/r) : U_p \to {\bf B}  . $$
 Let \(\omega \) be a smooth non negative \((1,1)\)-form with the following properties:
\begin{itemize}
\item the support of $\omega$ is contained in \(U_r\), 
\item $\omega$ is equal to  \(  i (dx_p \wedge d\overline{x_p} + dy_p \wedge d\overline{y_p} ) \) for   \( (x,y) \in U_{r/2}\), 
\item $\omega$ is bounded by \(i (dx_p \wedge d\overline{x_p} + dy_p \wedge d\overline{y_p} )\) for \( (x,y) \in U_r\).
\end{itemize}
For any \( \mathcal P_{\mathcal F, r}\)-harmonic current \(T_r \), we have 
\[  T_r (\omega) \leq \frac{4}{r^2} \int_{U_r}  T \wedge  i (dx \wedge d\overline{x} + dy \wedge d\overline{y} )\leq 4 I_{T_r}(r).  \]
Denoting by \(\Theta_m\) the curvature form of a hermitian metric \(m\) on \(N_{\mathcal F}\), and letting \( \eta :=  \Theta_m - M \omega \), we get from Lemma \ref{l: estimation} that
\[ T_r (\eta) >0 \text{ for any }  \mathcal P_{\mathcal F, r}\text{-harmonic current } T_r .\]
Theorem \ref{c: positivity II} then follows from Lemma \ref{l: Hahn-Banach} applied with $\mathcal P = \mathcal P_{\mathcal F, r}$.
 \end{proof}

\section{Thin sets in the complex plane} \label{s: thin sets}

In this section, we study some geometrical/stochastical properties of closed sets $\Lambda$ in a Riemann surface \(C\). For every set \( V\subset C \), and any path \(\gamma : [0,+\infty) \rightarrow C\) starting at a point \(\gamma (0) \in V\), denote by \( T_V( \gamma)\) the largest time \(\gamma\) remains in \(V\), namely 
\[ T_V (\gamma) := \sup ( t\geq 0 \ |\ \gamma ([0, t] ) \subset V ) . \]
Let $\mathbb P^x$ be the Wiener measure on the set of continous paths starting at $x$.

\begin{definition} [Doob] 
A closed subset \(\Lambda \subset C\) is \emph{thin} if for every point \( x\in \Lambda\), a Brownian trajectory starting at \(x\) almost surely leaves \(\Lambda\) at arbitrarily small times (that is $T_\Lambda(\gamma) =0$ for every $x \in \Lambda$ and $\mathbb P^x$-almost every $\gamma$). 
\end{definition}

We refer to the book \cite[Chapter II, p.79]{BG} for more on this notion. Note that this definition does not depend on the choice of the hermitian metric on \(C\), by conformal invariance of the Brownian motion. In the sequel, for \(x \in C\) we will denote by \(\Gamma^x\) the set of continuous paths \(\gamma : [0,\infty) \rightarrow C\) such that \(\gamma (0)=x\), and for every \(t\geq 0\)
\[E^{x,t}_{\Lambda} := \{ \gamma \in \Gamma^x, T_{\Lambda} (\gamma) \geq t \}.\]

\begin{lemma}\label{lemma: positiveLeb}
If \(\Lambda \) has zero Lebesgue measure then it is thin.  
\end{lemma} 

\begin{proof}
Let us prove that if $\Lambda$ is not thin, then its Lebesgue measure is positive. By assumption there exist $x \in \Lambda$ and $t>0$ such that  $\mathbb P^x ( E_{\Lambda}^{x,t} ) > 0$. Let $\mathbb P^{x,t}_\Lambda$ be the restriction of $\mathbb P^x$  to $E_{\Lambda}^{x,t} $. Then $(\pi_t)_* \mathbb P^{x,t}_\Lambda \leq (\pi_t)_* \mathbb P^x$, where $\pi_t : \Gamma^x \to \Lambda$ is defined by $\pi_t(\gamma) = \gamma(t)$. Since $(\pi_t)_* \mathbb P^x$ is absolutely continuous with respect to the Lebesgue measure on $C$ and since the support of $(\pi_t)_* \mathbb P^{x,t}_\Lambda$ is included in $\Lambda$, the Lebesgue measure of $\Lambda$ is positive. \end{proof}

The following result is presumably well-known, but we provide a complete proof since we do not know any reference on this.

\begin{proposition}\label{p: thin implies large first eigenvalues}
A compact subset \( \Lambda\subset \C\) is thin if and only if $\Lambda$ possesses relatively compact open neighborhoods $V$ with smooth boundaries whose first eigenvalue with respect to the Dirichlet problem is arbitrarily large.
\end{proposition} 

Let us recall some facts about the Dirichlet problem, we refer to \cite[Chapter I]{Chavel} for a general account.
Let $D$ be a bounded domain in $\C$ with smooth boundary.
An eigenvalue for the Dirichlet problem on $D$ is a real number $\lambda$ such that $\Delta \varphi +  \lambda \varphi =0$ for some bounded $C^2$ function $\varphi$ with zero boundary values. These eigenvalues form a sequence of positive numbers which tends to infinity, let $\lambda_1(D)$ denote the first (smallest) eigenvalue. It  is inclusion decreasing, that is $\lambda_1(D_1) \leq \lambda(D_2)$ if $D_2 \subset D_1$. 

We shall use two features concerning $\lambda_1$. The first one is that \(\mathbb E^x (\exp (\lambda T_D ) ) < \infty \) for every \(x\in D \) if and only if \(\lambda_ 1 (D) \geq \lambda\), see \cite[Section 3]{Sullivan_positivity}. The second one is that the norm of the diffusion operator $P^t_D$ (see for instance Equation (\ref{eq: HS2}) below) is equal to $e^{-t \lambda_1(D)}$, see \cite[Section 4.7]{PS}.

If $V$ is an open subset of $\C$ with smooth boundary, then $\lambda_1(V)$ is the infimum of $\lambda_1(D)$ where $D$ runs over the connected components of $V$. In particular, the norm of $P^t_V$ is equal to $e^{-t \lambda_1(V)}$.

Let us introduce open neighborhood of $\Lambda$ which will be used in the proofs of Propositions \ref{p: thin implies large first eigenvalues} and  \ref{p: technical characterization}. For every $\epsilon >0$, we define \( \Lambda^\varepsilon :=\{ d_\C(\cdot , \Lambda) <\varepsilon\} \) and set an open neighborhood \(V_\varepsilon\) of $\Lambda$ with smooth boundary satisfying 
\begin{equation}\label{eq: V}
\Lambda \subset V_\varepsilon \subset \Lambda^\varepsilon .
\end{equation}

\begin{proof}[Proof of Proposition \ref{p: thin implies large first eigenvalues}]
Assume that $\Lambda \subset \C$ is a thin compact set. 
We claim that for every \(t>0\) and \(\delta>0\), there exists \(\varepsilon>0\) such that 
\begin{equation} \label{eq: quantitative thin property} \mathbb P^x ( T_{\Lambda^\varepsilon}  >t ) \leq \delta. \end{equation} 
Indeed, assume to the contrary that \eqref{eq: quantitative thin property} does not hold: there exist \(t>0\) and \(\delta>0\) so that for every \(\varepsilon >0\), there exists \(x_\varepsilon \in \Lambda^\varepsilon\) such that 
\begin{equation} \label{eq: contradiction} \mathbb P^{x_\varepsilon} ( T_{\Lambda^\varepsilon}  >t ) > \delta.\end{equation} 
By compactness of \(\Lambda\), we can find a sequence of positive numbers  \(\varepsilon_n \)  that tends to \(0\)  and such that \(x_{\varepsilon_n} \) tends to some \(x\in \Lambda\) when \(n\) goes to infinity.

The triangular inequality immediately yields that 
\begin{equation}\label{eq: triangular inequality consequence} x-y + E _{\Lambda^\eta} ^{y,t} \subset E_{\Lambda^{\eta+ |x-y|}} ^{x,t} ,\end{equation}
where  \(z+E _{\Lambda^\eta} ^{y,t}\) denotes the set of paths of the form \(z + \gamma (t) \) with \(\gamma \in E _{\Lambda^\eta} ^{y,t}\).
Observe that \(\mathbb P^{x_{\varepsilon_n}} (E_{\Lambda^{\varepsilon_n}} ^{x_{\varepsilon_n}, t })=\mathbb P^{x_{\varepsilon_n}} ( T_{\Lambda^{\varepsilon_n}}  >t ) > \delta\) by  \eqref{eq: contradiction}. Hence, together with \eqref{eq: triangular inequality consequence} and the equivariance of the Wiener measures \(\mathbb P^x\) with respect to translations, we get 
\[ \mathbb P^x (E ^{x,t}_{\Lambda^{ \varepsilon_n + |x-x_{\varepsilon_n}|}} ) \geq \mathbb P ^x (x-x_{\varepsilon_n}+  E ^{x_{\varepsilon_n},t} _{\Lambda^{ \varepsilon_n}} )=\mathbb P^{x_{\varepsilon_n}} ( E ^{x_{\varepsilon_n}, t}_{\Lambda^{ \varepsilon_n}} )>\delta.\] 
Denoting by \(\eta_n= \varepsilon_n + |x-x_{\varepsilon_n}|\) and taking if necessary a subsequence so that \(\eta_n\) is decreasing, we get that \(E ^{x, t} _{\Lambda^{\eta_n}}\) is also decreasing for inclusion, and this yields 
\[ \mathbb P^x (\cap _ n  E^{x,t}_{\Lambda^{ \eta_n}} ) \geq \delta .\]
But since $\Lambda$ is closed, the intersection \( \cap _ n  E^{x,t}_{\Lambda^{ \eta_n}}\) is the set of continuous paths \(\gamma: [0,+\infty) \rightarrow \C\) so that \(\gamma (0) = x\) and \(\gamma ([0,t]) \subset \Lambda\). This contradicts the thin property, hence \eqref{eq: quantitative thin property} holds.

We now obtain, by iterating \eqref{eq: quantitative thin property} and using the Markov property, that for every  \(k \geq 1 \) and every \(x\in \Lambda ^{\varepsilon}\), 
\[ \mathbb P^x ( T_{\Lambda^\varepsilon}  >kt ) \leq \delta^k.\]
In particular for every \(\lambda >0\) and \(x\in \Lambda^{\varepsilon}\), we get 
\[ \mathbb E^x (\exp (\lambda T_{\Lambda^\varepsilon} ) ) \leq \sum _{k\geq 0} \int _{kt < T_{\Lambda^\varepsilon} \leq (k+1) t}  \exp (\lambda T_{\Lambda^\varepsilon} ) d\mathbb P^x \]
\[ \leq  \text{cst} + \sum _{k \geq 1} \delta ^k \exp (\lambda (k+1) t) <+\infty\]
if \(\log \delta + \lambda t <0\). This condition can be fulfilled by appropriately choosing the constants \( t, \delta >0\). For the corresponding value of \(\varepsilon\)  we get the convergence of \(\mathbb E^x (\exp (\lambda T_{\Lambda^\varepsilon} ) )\). Now let $D$ be a connected component of $V_\epsilon$ defined in (\ref{eq: V}). Since $D \subset \Lambda^\varepsilon$, we get $T_D \leq T_{\Lambda^\varepsilon}$, hence \(\mathbb E^x (\exp (\lambda T_D ) )\) converges for every \(x\in D \). That proves \(\lambda_ 1 (D) \geq \lambda\) by \cite[Section 3]{Sullivan_positivity}, and thus \(\lambda_ 1 (V_\varepsilon) \geq \lambda\).

Conversely, let $\Lambda \subset \C$ be a compact subset having relatively compact open neighborhoods with smooth boundaries whose first eigenvalue is as large as we want. By the monotonicity property of the first eigenvalue, \(\lambda_1(V_\varepsilon)\) tends to \(+\infty\) when \(\varepsilon\) tends to zero.

We proceed by contradiction assuming that $\Lambda$ is not thin. First we fix a relatively compact open neighborhood \(V\) of \(\Lambda\) with smooth boundary that contains every  \(V^\varepsilon\). Let us take the notations of the proof of Lemma  \ref{lemma: positiveLeb}. For every $x \in \Lambda$ and $t >0$ such that $ \mathbb P^x(E_{\Lambda} ^{x,t}) >0$, let $\mu_{\Lambda}^{x,t} := (\pi_t)_* \mathbb P^{x,t}_\Lambda$ and let $q_\Lambda (x,\cdot,t)$ be the density of $\mu_{\Lambda}^{x,t}$ with respect to the Lebesgue measure on $\mathbf C$. The latter is bounded above by the heat kernel $p(x,\cdot, t)$ for the euclidian distance on $\C$, which is a continuous function. Let 
\begin{equation}\label{eq: HS} P_{\Lambda}^t(f)(x) := \int_{\Lambda} f(y) d \mu_{\Lambda} ^{x,t}(y)  ,
\end{equation}
this defines a compact self-adjoint operator of $L^2(V)$. 
 Similarly,  let $\mu_{V_\epsilon}^{x,t} := (\pi_t)_* \mathbb P^{x,t}_{V_\epsilon}$ and $q_{V_\epsilon} (x,\cdot,t)$ be the density of $\mu_{V_\epsilon}^{x,t}$ with respect to the Lebesgue measure. Note that $q_{\Lambda} \leq q_{V_\epsilon} \leq p$ since $E_{\Lambda} ^{x,t} \subset E_{V_\epsilon}^{x,t}$. Let us define
\begin{equation}\label{eq: HS2}  P_{V_\epsilon}^t (f)(x) :=  \int_{V^\varepsilon} f(y) d \mu_{V_\epsilon}^{x,t} (y)  
 \end{equation}
and prove that $P_{V_\epsilon}^t$ converges to $P_{\Lambda}^t$ in the space of operators of $L^2(V)$. First observe that one can integrate over $V$ instead of $\Lambda$ and $V_\epsilon$ in (\ref{eq: HS}) and (\ref{eq: HS2}) without modifying the definitions. Now for every $f \in L^2(V)$ of norm one, 
\begin{equation}\label{eq: strongconv}
 \norm { P_{V_\epsilon}^t (f) - P_{\Lambda}^t (f) }^2 \leq  \iint_{V \times V} (q_{V_\epsilon} - q_{\Lambda} )^2 dxdy \leq M  \iint_{V \times V} (q_{V_\epsilon} - q_{\Lambda} ) dxdy ,  
\end{equation}
where the first inequality uses Cauchy-Schwarz, and the last inequality uses $q_{\Lambda} \leq q_{V_\epsilon} \leq p$, $M$ being an upper bound of $p(\cdot,\cdot,t)$ on  $V \times V$. Since $\cap _{\epsilon >0} E_{V_\epsilon}^{x,t} = E_\Lambda^{x,t}$  for every $x \in \Lambda$, the right hand side term of Equation (\ref{eq: strongconv}) tends to zero as $\epsilon$ tends to zero. 
In particular, the norm of $P_{V_\epsilon}^t$ tends to the norm of $P_\Lambda^t$, hence $\lambda_1(V_\epsilon)$ is bounded, a contradiction.  The compact set $\Lambda$ is thus thin, and the proof is complete. 
 \end{proof}

We will need the following rather technical result. 

\begin{proposition} \label{p: technical characterization}
A closed subset \(\Lambda \subset \D\) is thin  if and only if there exists a sequence of smooth functions \( f_n : \D \rightarrow \R\) such that, in restriction to any compact subset of \( \Lambda \), we have, uniformly:
\begin{itemize} 
\item \(f_n\) converges to \(0\),
\item \( \Delta f_n\) tends to \( +\infty\),
\item \( \vert \nabla f_n \vert  ^2 = o (\Delta f_n )\).
\end{itemize} 
\end{proposition}

\begin{proof} Since the desired convergence of the sequence is in restriction to any compact, we can assume that \(\Lambda\subset \D\) itself is compact.  
We first prove that if  \(\Lambda\) is thin then there exists such a sequence of functions. We borrow notations from the proofs of Lemma \ref{lemma: positiveLeb} and Proposition \ref{p: thin implies large first eigenvalues}. Fix  \(\lambda >0\), and let 
\(\varepsilon = \varepsilon (\lambda) >0\) be small enough so that \(\lambda _1(V_\varepsilon) >\lambda\). Let  
\begin{equation} \label{eq: expectation of exponentiel of hitting time} \psi _{\lambda, \varepsilon} (x) := \mathbb E^x \left( \exp ( \lambda T_{V_{\varepsilon}} ) \right)  : V_\varepsilon \rightarrow \R , \end{equation}
it is well defined since \( \lambda_1(V_\varepsilon) > \lambda \).  By \cite[Section 3]{Sullivan_positivity}, $(\psi _{\lambda, \varepsilon})_\lambda$ satisfies 
\begin{equation}
 \left\{
    \begin{array}{l}
        \Delta \psi_{\lambda, \varepsilon} + \lambda \psi_{\lambda, \varepsilon } =0 \\
        \psi_{\lambda, \varepsilon} = 1 \text{ on } \partial V_{\varepsilon}
    \end{array}
\right.
\end{equation}
and uniformly converges to \( 1\) on \(\Lambda\)  when $\lambda$ tends to $\infty$ (hence when \(\varepsilon\) tends to \(0\)), because $\Lambda$ is thin. Let 
\begin{equation} \label{eq: function phi} \varphi_{\lambda, \varepsilon } : = -\log \psi_{\lambda, \varepsilon}.\end{equation} 
The sequence $( \varphi_{\lambda, \varepsilon})_\lambda$  uniformly converges to 0 on \(\Lambda\), and we have 
\begin{equation} \Delta \varphi_{\lambda, \varepsilon} = - \frac{\Delta \psi_{\lambda,\varepsilon}} {\psi_{\lambda,\varepsilon}} + \vert \nabla \varphi_{\lambda,\varepsilon} \vert ^2 = \lambda +  \vert \nabla \varphi_{\lambda,\varepsilon} \vert ^2 . \end{equation} 
So we deduce 
\begin{equation}\label{eq: bounds varphi} \Delta \varphi_{\lambda, \varepsilon}\geq \lambda \text{ and }   \vert \nabla \varphi_{\lambda, \varepsilon} \vert ^2 \leq \Delta \varphi_{\lambda,\varepsilon} .\end{equation}
Now for every positive integer \(n\), define  
\begin{equation} \label{eq: fn} f_n := \frac{1}{n} \varphi _{n^2, \varepsilon_n} \end{equation}
where \(\varepsilon_n\) is small enough so that \( \lambda_1 ( V_{\varepsilon_n} ) > n^2 \). Equation \eqref{eq: bounds varphi} then implies 
\[ \Delta f_n \geq n \text{ and } \vert \nabla f_n \vert  ^2 \leq \frac{1}{n} \Delta  f_n , \] 
which concludes the first part of the proof. 

Conversely, consider a sequence of functions \(f_n\) as in the statement of Proposition \ref{p: technical characterization}, and let \(\lambda > 0 \). For \(n\) large enough, there exists a relatively compact open neighborhood \(V_n\) of \(\Lambda\) on which 
\begin{equation} \label{eq: condition}  \Delta f_n \geq 2\lambda \text{ and } \Delta f_n \geq 2 \vert \nabla f_n \vert ^2 . \end{equation} 
We can assume that \( V_n \) has smooth boundary. Denote \(\psi := \exp ( - f_n) \). The function \(\psi\) is positive on $V_n$ and satisfies 
\[ -\Delta \psi = \psi \left( \Delta f_n - \vert \nabla f_n \vert ^2 \right) \geq \lambda \psi.\]
This inequality, together with Lemma \ref{l: reciproque} below, implies $\lambda_1(V_n) \geq \lambda$. We then conclude by applying Proposition \ref{p: thin implies large first eigenvalues}. \end{proof}

\begin{lemma} \label{l: reciproque}
Let \( V  \subset \C\) be a relatively compact open set  with smooth boundary, and let \( \psi : V\rightarrow \R\) be a positive function such that \( - \Delta \psi \geq \lambda \psi\) for some constant \(\lambda\). Then, \(\lambda_1(V)  \geq \lambda\).  
\end{lemma} 

 \begin{proof} Let $D$ be a connected component of $V$ and let \(\chi : D \rightarrow \R\) be an eigenfunction for the first eigenvalue of $D$. We thus have  
\begin{equation*}
 \left\{
    \begin{array}{l}
        \Delta \chi + \lambda_1(D) \chi =0 \\
        \chi = 0 \text{ on } \partial D.
    \end{array}
\right.
\end{equation*}
By Courant's nodal domain theorem \cite[Section I.5]{Chavel}, the function $\chi$ does not vanish on $D$, we can assume that $\chi$ is positive.
 Green's formula reads 
\begin{equation*} \label{eq: Green}  \int_D  \left( \chi \Delta \psi   -   \psi \Delta \chi \right) dv = \int_{\partial D} \left( \chi \nabla \psi - \psi \nabla\chi\right) \cdot n_{ext} ,
\end{equation*}
where \( n_{ext}\) is the exterior normal vector to \(\partial D\) and \( v\) is the Lebesgue measure on \(\C\). Since \(\chi \) is positive on \(D\) and vanishes on \(\partial D\), the exterior derivative \( \nabla \chi \cdot n_{ext}\) is non positive, hence $\int_D  \left( \chi \Delta \psi   -   \psi \Delta \chi \right) dv \geq 0$. Moreover, 
\[ \int _D   \left( \chi \Delta \psi   -   \psi \Delta \chi \right) dv \leq \int_ D \left( -\lambda \chi \psi +\lambda _1(D) \chi \psi\right) dv =  (-\lambda + \lambda_1(D) ) \int_D \chi \psi dv . \]
That proves $\lambda_1(D) \geq \lambda$ since \(\chi \) and \(\psi\) are positive on \(D\). Taking into account every connected component $D$ of $V$, we get $\lambda_1(V) \geq \lambda$ as desired.
\end{proof}

\section{Positivity of the normal bundle in all directions at each point of the limit set}\label{s: positivity all directions}

The goal of this part is to use the thin property to gain positivity for the normal bundle $N_\FF$ in all directions at each point of the limit set $\mathcal L$. The proof is inspired by Brunella's article \cite{BrunellaToulouse}, see also \cite{Canales}.

We denote by \(\mathcal L\) the limit set of \(\FF\). Let \(\Lambda \subset  \D\) be the image of \(\mathcal L \cap U\) by a local first integral \( t : U\rightarrow \D\); we call such a set \(\Lambda\) a transversal set of \(\mathcal L\). We will say that \(\mathcal L\) is thin if any transversal set \(\Lambda\) is thin. Recall that the thin property is a local property, so using the minimality, \(\mathcal L\) is thin if and only if some of its transversal set is thin.  

\begin{theorem}(Improvement of Theorem \ref{c: positivity II}) \label{c: metric of positive curvature III}
Assume that \(\mathcal L \) is thin. For every $p \in \text{sing}(\mathcal F)$, let $(x_p,y_p) : U_p \to \bf B$ be linearization coordinates provided by Theorem \ref{c: positivity II} and let \(M >0\). There exists a hermitian metric \(m\) on \(N_{\mathcal F}\) such that 
\begin{enumerate}
\item  the curvature of $m$ is positive on \(T_q S\) for every \(q \in \mathcal L \), 
\item  the curvature of $m$ is bounded from below by \( M (i dx_p\wedge d\overline{x_p} + i dy_p \wedge d\overline{y_p} )  \) on \(  U_p \) for every \(p\in \text{sing} (\mathcal F)\).
\end{enumerate}
\end{theorem}

Let \( U_ p(r) := \{ |x_p|^2+|y_p|^2< r^2 \} \) and $U(r) :=  \bigcup_p U_p (r)$. Let \( (V_j)_{j \in J}\) be a finite covering of $S \setminus  U(1/\sqrt {16})$ by foliated charts such that $V := \bigcup_{j} V_j$ does not intersect $U(1/\sqrt{32})$. In particular, $\partial V \subset U(1/\sqrt 8)$. These special properties on  $(V_j)_{j \in J}$ will be used in Section \ref{s: convexity}.

Let \( \rho _j : S \rightarrow {\bf R^+}\) be smooth functions whose support is contained in \(V_j\) such that $\sum_j \rho_j$ does not vanish on $V$. We can choose \(\rho_j\) satisfying 
\begin{equation} \label{eq: estimates partition of unity} 
 \vert { D^k \rho _j } \vert \leq C \rho _j^{1/2}  
\end{equation}
for every $j \in J$ and $k=1,2$, where \(D^k\rho_j \) denotes the \(k\)-th derivative of $\rho_j$.

Let \( (z_j, t_j) : V_j \rightarrow {\bf D}\times {\bf D}\) be foliated coordinates, and let \( m \) be the hermitian metric on \(N_{\mathcal F}\) constructed in Theorem \ref{c: positivity II}. The curvature of $m$ is positive in restriction to \(\mathcal P_{\mathcal F}\), hence we have
\[ m = \exp (- \varphi_j) \ |dt_j| \textrm{ on } V_j ,\]
where \(\varphi_j \) is a smooth strictly subharmonic function along the leaves.

\begin{proof}[Proof of Theorem \ref{c: metric of positive curvature III}]
For every \(j\in J\), let \(\Lambda_j\subset {\bf D}\) be the image of \( \mathcal L\cap V_j\) by the map \(t_j: V_j \rightarrow {\bf D}\). Since the set \(\Lambda_j\) is thin, there exists by Proposition \ref{p: technical characterization} a sequence of smooth functions \(\{ f_n^j\} _n :{\bf D} \rightarrow {\bf R} \)  satisfying the following properties uniformly in restriction to \( \Lambda_j \)   when \(n\) tends to infinity:
\begin{itemize} 
\item[(i)] \(f_n^j\) converges to \(0\),
\item[(ii)] \( \Delta_{t_j} f_n^j = (f_n^j)_{t_{j}\overline{t_{j}}} \) tends to \( +\infty\),
\item[(iii)] \( \vert \nabla_{t_j} f_n^j \vert  ^2 = o (\Delta_{t_j} f_n ^j )\).
\end{itemize} 
Let us define 
\[ f_n := \sum_{j\in J} \rho_j  \, f_n^j \circ t_j  : S \to \R^+ \]
and the hermitian metrics on \( N_{\mathcal F} \) 
\[ m_n := \exp (- f_n) m .\] 
Note that $m_n = m$ on $S \setminus V$ since the support of $f_n$ is included in $V$. In particular, by Theorem \ref{c: positivity II}, the curvature of $m_n$ is bounded from below by \( M (i dx_p\wedge d\overline{x_p} + i dy_p \wedge d\overline{y_p} )  \) on \( U_p  \setminus V \) for every \(p\in \text{sing} (\mathcal F)\).

It remains to prove that, for $n$ large enough, the curvature of $m_n$ is positive on \(T_q S\) for every \(q \in \mathcal L \cap V \).
Let \(j_0 \in J \) and $J_0 := \{ j \in J , V_j \cap V_{j_0} \neq \emptyset \}$. We have on $V_{j_0}$:
$$ m_n = \exp ( -\phi_n ^{j_0} ) |dt_{j_0}| , \textrm{ where }  \phi_n ^{j_0}:= \varphi_{j_0}+f_n  .   $$
Let us  introduce the following derivatives on $V_{j_0}$:
$$ \alpha_n := (\phi_n^{j_0} )_{z_{j_0}\overline{z_{j_0}}} \ \ , \ \ \beta_n := (\phi_n^{j_0} )_{z_{j_0}\overline{t_{j_0}}} \ \ , \ \ \gamma_n :=  (\phi_n^{j_0} )_{t_{j_0}\overline{t_{j_0}}} .$$
We have to show that  \( \alpha_n , \gamma_n \), and \( \alpha_n \gamma_n - \beta_n^2 \) are positive on $\mathcal L$. In the remainder $f_n^j$ simply stands for $f_n^j \circ t_j$, and we denote 
$$ \alpha := (\varphi_{j_0})_{z_{j_0}\overline{z_{j_0}}} \ \ , \ \  \beta := (\varphi_{j_0})_{z_{j_0}\overline{t_{j_0}}} \ \ , \ \ \gamma := (\varphi _{j_0})_{t_{j_0}\overline{t_{j_0}}} . $$
The notation \( o_{\mathcal L} (1)\) refers to a function that tends to zero when the argument tends to a point of  \( {\mathcal L} \).

\vspace{0.2cm} 

\textit{Positivity of  $\alpha_n$:} Since \( f_n^j\) only depends on \( t_{j_0}\), we have on $V_{j_0}$:  
\begin{equation}\label{eq: estimates1}
  \alpha_n  =  \alpha + (f_n)_{z_{j_0}\overline{z_{j_0}}}  = \alpha + \sum _{j\in J_0}( \rho_j  ) _{z_{j_0}\overline{z_{j_0}}} f_n^j   = \alpha  +  o_{\mathcal L} (1) ,
\end{equation}
where the last equality comes from property $(i)$. Since $\alpha$ is a positive  function on $\mathcal L$, the function $\alpha_n$ is positive on $\mathcal L$ for $n$ large enough.

\vspace{0.2cm} 

In order to show the positivity of $\gamma_n$  and $\alpha_n \gamma_n - \beta_n^2$ on $\mathcal L$, we introduce 
\[ \Delta_n^{j_0} : = \sum_{j \in J_0} \rho_j (f_n^j  ) _{t_{j_0} \overline{t_{j_0}}} : V_{j_0} \to \R  ,\] 
which tends to \(+\infty\) at each point of $\Lambda_{j_0}$ in the support of some \(\rho_j\). By  properties $(i), (ii), (iii)$ and Equation (\ref{eq: estimates partition of unity}), we get 
 $$ \sum_{j\in J_0} \vert D^2 \rho_j \vert f_n^j \leq C \sum_{j\in J_0} \rho_j^{1/2} f_n^j \leq C (\sum_{j\in J_0} \rho_j f_n^j)^{1/2} ( \sum_{j\in J_0} f_n^j  )^{1/2} = o_{\mathcal L} ((\Delta_n^{j_0})^{1/2}) ,  $$
$$ \sum_{j\in J_0} \vert D^1 \rho_j \vert   \vert \nabla_{t_j} f_n^j  \vert \leq C (\# J_0) ^{1/2}  ( \sum_{j\in J_0} \rho_j \vert \nabla_{t_j} f_n^j \vert ^2)^{1/2} = o_{\mathcal L}((\Delta_n^{j_0})^{1/2}) .$$

\vspace{0.2cm} 

\textit{Positivity of $\gamma_n$:} the preceding estimates and the computation
\[   \gamma_n  =  \gamma + \Delta_n ^{j_0}  +  \sum _{j\in J_0} \left(  (\rho_j)_{t_{j_0}\overline{t_{j_0}}}  f_n^j + 2\Re \left((\rho_j)_{t_{j_0}}  (f_n^j )_{\overline{t_{j_0}}} \right)  \right)   \]
imply
\begin{equation} \label{eq: estimates4} 
\gamma_n = \gamma + \Delta_n ^{j_0} + o_{\mathcal L} \left( ( \Delta_n^{j_0} ) ^{1/2} \right)   .
\end{equation}
Since $\Delta_n ^{j_0}$ tends to $+ \infty$ on $\mathcal L$,  $\gamma_n$ is positive on $\mathcal L$ for $n$ large enough. 

\vspace{0.2cm} 
 
\textit{Positivity of $\alpha_n \gamma_n - \beta_n^2$:} here we have
\begin{equation}\label{eq: estimates3}
\beta_n  =  \beta + \sum _{j\in J_0} \left( (\rho_j ) _{z_{j_0}\overline{t_{j_0}}}  f_n^j  + (\rho_j )_{z_{j_0}} (f_n^j)_{\overline{t_{j_0}}} \right) =  \beta  + o_{\mathcal L} \left( (\Delta_n^{j_0})^{1/2} \right) .
 \end{equation}
By using  \eqref{eq: estimates1}, \eqref{eq: estimates3} and \eqref{eq: estimates4}, we obtain that  $\alpha_n \gamma_n - \beta_n^2 = \alpha \Delta_n^{j_0} + o ( \Delta_n^{j_0} )$, which tends to $+ \infty$  on $\mathcal L$. That completes the proof of Theorem \ref{c: metric of positive curvature III}.\end{proof} 

\section{Convexity of the complement of the limit set: proof of Theorem \ref{t: convexity II}}  \label{s: convexity}

We prove in this Section that $S \setminus \mathcal L$ is strongly pseudoconvex, hence it is a modification of a Stein manifold by Grauert's Theorem \cite[Theorem 2]{Grauert}. Namely,  we have to prove that there exists a proper and strictly plurisubharmonic function $h : \mathcal V \setminus \mathcal L \to \R$ where $\mathcal V$ is a neighborhood  of $\mathcal L$ in $S$. 
We shall follow Brunella's construction \cite[Section 3.1]{BrunellaToulouse} in the non-singular setting, and perform an additional analysis near the singularities of $\FF$.

\subsection{Introduction of $m$-functions}

\begin{definition}\label{def: mfunction} Given a hermitian metric \(m \) on \(N_{\mathcal F}\), and an open set \( Y \subset S^* = S\setminus \text{sing} (\FF)\), a function \( f: Y \setminus \mathcal L\rightarrow {\bf R} \) is called a \(m\)-function, if at any point \( p \in Y \cap \mathcal L\), there exists a local submersion \( t :  W_p \rightarrow {\bf C} \) defining the foliation \(\mathcal F\) on a neighborhood \( W_p \) of \(p\), such that 
\begin{equation} \label{eq: log distance} f := \varphi - \log d_{\bf C}( t , \Lambda) + o_{\mathcal L \cap Y}(1) \textrm{ on } W_p \cap Y \setminus \mathcal L , \end{equation} 
 where \(\Lambda = t (\mathcal L)\), \(d_{\bf C}\) is the euclidean distance on ${\bf C}$ and \(m= \exp (-\varphi ) |dt|\). 
\end{definition} 

Note that \( - \log d_{\bf C} (t , \Lambda) \) is plurisubharmonic on $U \setminus \mathcal L$ since it is equal to  \(\sup_{\xi \in \Lambda} - \log {| t -\xi |} \). In particular, $dd^c f \geq dd^c \varphi$. 

\begin{lemma} (\cite[Lemma 3.2]{BrunellaToulouse}) \label{lemme: BruTou}
Two \(m\)-functions \(f: Y \rightarrow {\bf R}\) and \(f ' : Y' \rightarrow {\bf R} \) satisfy
\[  f - f ' = o _{\mathcal L \cap Y\cap Y'} (1) .\]
\end{lemma}

\begin{proof}  Let \(p\in Y\cap Y' \cap \mathcal L\). By definition, there exists a neighborhood \(W_p\) of \(p\) and two submersions \( t, t' : W_p \rightarrow {\bf C} \) defining the foliation \(\mathcal F\) such that if \( \Lambda = t( \mathcal L \cap W_p) , \ \Lambda' = t' (\mathcal L \cap W_p)\) and 
\( m = \exp (-\varphi ) |dt|  = \exp (-\varphi ') |dt '|\), then
\[ f'- f  = \log \left( \left \lvert \frac{dt' }{dt}\right\rvert \cdot \frac{d_{\bf C} ( t , \Lambda) }{d_{\bf C} ( t' , \Lambda') } \right) + o_{\mathcal L \cap Y \cap Y'}(1) .  \]
However,  
\[   d_{\bf C} ( t' , \Lambda' ) =  \left \lvert \frac{dt' }{dt} \right \rvert \cdot d_{\bf C} ( t , \Lambda) \cdot (1+o_{\mathcal L \cap Y\cap Y'}(1))  \]
so the claim follows. \end{proof}

\begin{definition}\label{def: constant}
Let \(p\in \text{sing} (\FF)\) and \(U_p\) be linearization coordinates near $p$ provided by Theorem \ref{c: metric of positive curvature III}. The foliation $\FF$ is thus defined in these coordinates by $\omega = axdy - by dx$. Let \( E\) be the elliptic curve defined as the quotient of the restriction of \(\FF\) to the complement of the two separatrices $\{ xy =0 \}$ in  \(U_p \). Let $I$ be the quotient map 
$$ I : (x,y) \in U_p \setminus \{ xy = 0\} \mapsto \log {y^a \over x^b} \in E. $$
 We fix a non zero holomorphic \(1\)-form \(\eta\) on \(E\) and denote by \(\Lambda\) the \(I\)-image of  \(\mathcal L\cap U_p \setminus \text{separatrices} \). \end{definition}

\begin{lemma} \label{l: third step}
Let \(m\) be a smooth hermitian metric on the restriction of \( N_{\FF}\) to \(U_p\), that we write in the form 
\begin{equation}\label{eq: expression m} m= \exp (-\xi ) |I^* \eta| ,\end{equation} 
where \( \xi\) has some logarithmic singularities on the separatrices. Let us define  \[F_p := \xi  - \log d_E(I, \Lambda)  .\]
\begin{enumerate}
\item $F_p$ is a \(m\)-function near any point of \(\mathcal L\cap U_p \setminus \text{separatrices} \).

\item Near a point of a separatrix, $F_p$ is not a \(m\)-function, but differs from a genuine \(m\)-function by a bounded function. More precisely, there exists a constant $C_p$ such that for every compact set \( K\subset U_p \setminus \{p\}\), every \(m\)-function \(f: K\rightarrow {\bf R} \) and every \( \delta  >0\), there exists a neighborhood \(\mathcal V\) of \( K\cap \mathcal L\) such that \( \norm{F_p -f}_{\infty, \mathcal V}\leq C_p+ \delta \).
\end{enumerate}
 \end{lemma}

\begin{proof} At a neighborhood of a point of \(U_p \cap \mathcal L\) not belonging to the separatrices, a transverse coordinate for the foliation \(\mathcal F\) is locally given by the first integral \(I\). Hence $F_p$ is a $m$-function near those points. 

Denote by \( (S_k)_{k=1,2}\) the germs of separatrices passing through \(p\). On each \(S_k\), the closed meromorphic form \( I^* \eta \) has a pole and can be written locally \( \alpha_k \frac{dt_k}{t_k}\), where  \(t_k \) is a local submersion defining $\FF$ such that \( S_k =  \{ t_k=0 \} \). Let us be more explicit: for $(x_0 , 0) \in U_p$ with $x_0 \neq 0$,  a local submersion defining $\FF$ near $(x_0,0)$ is given by $t = y / x^{b/a}$. Then $I^* \eta = d \log {y^a \over x^b} = {1 \over xy}(ax dy - by dx)$ is equal to ${1 \over a} {dt \over t}$ outside $\{ t = 0 \}$, as claimed. Note that the local submersions $t_k$ cannot be globalized due to the effect of monodromy, but they are well-defined up to multiplication by a constant. We denote by \( \widetilde{\Lambda}_k \) the \(t_k\)-image of \( \mathcal L\). 

For $k = 1,2$, let us consider the function locally defined outside \(S_k\) by  
 \[ \widetilde{ u_k } := \log \left \lvert \frac{\alpha_k}{t_k} \right\rvert + \log \frac{d_{\bf C} (t_k, \widetilde{\Lambda}_k)}{d_E( I , \Lambda) }   . \] 
It is invariant by the holonomy map $h : t \mapsto e^{2i\pi {b \over a}t}$ produced by turning around the separatrix $\{ y = 0 \}$, this is due to the fact that $\widetilde{\Lambda}_k$ is $h$-invariant. Hence there exists a continuous function \( u _k : E \rightarrow {\bf R} \)  such that  \( \widetilde{u_k} = u_k \circ I\). Let \(C_p := \max (C_1,C_2)\), where \( C_k := \max _E |u_k|\).

Now if we write $m$  as \( \exp (- \psi _k ) |dt_k |\), then $$ \tilde f := \psi_k  - \log d_{\bf C} (t_k, \widetilde{\Lambda}_k ) $$ defines a \(m\)-function near every point of $S_k$. But \eqref{eq: expression m} and  \( I^* \eta = \alpha_k \frac{dt_k}{t_k}\) yield \( \xi - \psi_k  = \log\left( \frac{ |\alpha _k |} {|t_k|}\right) \). Hence  the function $$F_p = \xi  - \log d_E(I, \Lambda)$$ satisfies  \(F_p - \tilde f = \tilde u_k\), which  is bounded by $C_p$. One gets the stated property for an arbitrary $m$-function $f$ by using Lemma \ref{lemme: BruTou}. \end{proof}

\subsection{Proof of Theorem \ref{t: convexity II}}

 Let \(M\) be a constant larger than \(10 C_p\) (see Lemma \ref{l: third step} for the definition of $C_p$) for each \(p\in \text{sing} (\FF)\) and let \(m\) be a hermitian metric on \(N_\FF\) provided by Theorem \ref{c: metric of positive curvature III} for that constant $M$. The curvature of $m$ is thus positive on \(T_q S\) for every \(q \in \mathcal L \), and   is bounded from below by \( M (i dx_p\wedge d\overline{x_p} + i dy_p \wedge d\overline{y_p} )  \) on \(  U_p \) for every singular point $p$.
 
 We use the notations of Section \ref{s: positivity all directions}. 
On every $V_j$, $m = e^{- \varphi_j} \vert dt_j \vert$  where $\varphi_j$ is strictly plurisubharmonic. Let us consider a \(m\)-function 
$$ f_j = \varphi_j - \log d_{\bf C}( t , \Lambda) + o_{\mathcal L \cap V_j}(1) : V_j\setminus \mathcal L \to \R $$
 and set
\[ h_j := f_j + \varepsilon \rho_j . \]  
 We choose \(\varepsilon\)  small enough such that  \( h_j\) remains strictly psh  (this is possible since $dd^c \rho_j$ is bounded and $dd^c f_j \geq dd^c \varphi_j$) and such that $\epsilon \rho_j$ is smaller than $C_p /10 $ for every \(p\in \text{sing} (\FF)\). Now for every \(p\in \text{sing} (\FF)\), consider     
\[ h_p := F_p + 2 C_p - 4 C_p  (|x_p|^2+ |y_p|^2) : U_p \setminus \mathcal L\rightarrow {\bf R} , \]
where \( F_p \) is provided by Lemma \ref{l: third step}.
It is strictly plurisubharmonic by the choice of \(M\). Let us define 
\[  h : \mathcal V \setminus \mathcal L \rightarrow {\bf R} \ \ , \ \ h(q) := \sup _{ \nu } h_\nu (q) , \]
where the supremum is taken over $J_q := \{ \nu \in J,  q \in V_\nu \}$ and $P_q := \{ \nu \in  \text{sing}(\FF) ,  q\in U_\nu \}$. Note that $P_q$  has at most one element. Since $h$ is proper on $ \mathcal V \setminus \mathcal L$, the next proposition shows that $S \setminus \mathcal L$ is strongly pseudoconvex, which completes the proof of Theorem \ref{t: convexity II}. We shall follow the arguments of \cite[Lemma 3.3]{BrunellaToulouse}, but we have to adapt  them to take into account the singular set of $\FF$. This is where the delicate construction of $h_p$  (and its comparison with $h_j$ provided by Lemma \ref{lem: est}) enters the picture.

\begin{proposition} $h$ is continuous and strictly psh near every $q \in \mathcal V \setminus \mathcal L$.
\end{proposition}

\begin{proof}
There are several cases depending on $q$ described below. For each of them, the reader will verify that $h$ can be rewritten as  the supremum of a family of continuous and strictly plurisubharmonic functions all defined on some neighborhood $O_q$ of $q$. Recall that \( U_ p(r) = \{ |x_p|^2+|y_p|^2< r^2 \} \) so that $U_p = U_p(1)$.

 Let us begin with $q \in S \setminus U(1/\sqrt{16})$, for which there are three cases:

a) Some neighborhood  $O_q$ of $q$ satisfies $O_q \subset W$ or $O_q \cap W = \emptyset$ for every $W \in \{ V_j , j \in J \} \cup \{ U_p , p \in \text{sing}(\FF) \}$.  

b) The set $J_q^\partial := \{ j \in J , q \in \partial V_j\}$ is not empty. Recall that the support of the non negative smooth function \( \rho _j \) is contained in \(V_j\) and that $\sum_j \rho_j$ does not vanish on $\bigcup_j V_j$. Let us fix $j_0 \in J_q$ such that $\rho_{j_0}(q) > 0$ and let $j \in J_q^\partial$. Since $\rho_j(q) = 0$ and  $f_j - f_{j_0} = o_{\mathcal L \cap V_j \cap V_{j_0}}(1)$ by Lemma \ref{lemme: BruTou}, we get 
$$ f_j + \epsilon \rho_j < f_{j_0} + \epsilon \rho_{j_0} \ (\textrm{hence } h_j < h_{j_0}) \  \textrm{ on some } O_q \cap V_j \cap V_{j_0}  .$$ 

c)  The set $\{ p \in \text{sing}(\FF) , q \in \partial U_p \}$ is not empty, let $p$ denote its single element. The first item of Lemma \ref{lem: est}  below implies for every $j_0 \in J_q$:
 $$h_p < h_{j_0} - C_p  \  \textrm{ on some } O_q \cap  U_p \cap V_{j_0}  . $$

To finish it remains to consider $q \in  U_p(1/16)$ for some $p \in \text{sing} (\FF)$. For every $j \in J_q$, the second item of Lemma \ref{lem: est} implies  $$ h_j < h_p - 3 C_p/10 \textrm{ on some } O_q \cap U_p(1/\sqrt 8) \cap V_j  ,  $$
hence $h$ is simply equal to  $h_p$ on $U_p(1/\sqrt{16})$, and we are done.
\end{proof}

\begin{lemma} \label{lem: est}
There exists a neighborhood \( \mathcal V \) of \(\mathcal L\) such that for every $p \in \text{sing} (\FF)$, we have 
\begin{enumerate}
\item if $V_j \cap \partial U_p \neq \emptyset$, $h_p \leq h_j - 19 C_p / 10$ on $\mathcal V \cap V_j \cap \partial U_p$,
\item if $V_j \cap U_p(1/\sqrt 8) \neq \emptyset$, $h_j  \leq  h_p - 3 C_p / 10$ on  $\mathcal V \cap V_j \cap U_p(1/\sqrt 8)$.
\end{enumerate}
\end{lemma}

\begin{proof}
By applying Lemma \ref{l: third step} with $f = f_j$ and $\delta = C_p / 10$, we obtain 
\begin{equation*}\label{hphq1} 
h_p = F_p - 2 C_p \leq f_j -  19 C_p / 10     \textrm{ on }  \mathcal V \cap V_j \cap \partial U_p ,
\end{equation*}
the first point of Lemma \ref{lem: est} then follows from $f_j \leq h_j$ on $V_j$. 
For the second one, we first use Lemma \ref{l: third step} as before and then the upper bound $F_p \leq h_p - 3C_p /2$ on $U_p(1/\sqrt 8)$ to get 
\begin{equation*}\label{hphq2} 
 h_j = f_j + \epsilon \rho_j  \leq (F_p + 11 C_p / 10) + C_p / 10 \leq  h_p - 3 C_p / 10   \textrm{ on }  \mathcal V \cap V_j \cap  U_p(1/\sqrt 8) ,
 \end{equation*}
which completes the proof.
 \end{proof}

\section{The Julia set of a polynomial mapping is thin}

The present section is devoted to the proof of Theorem~\ref{t: Julia polynomial}. Actually, it is a particular case of the following statement.

\begin{proposition}\label{p:C-K}
Let $K\subset \C$ be a compact set, that coincides with the boundary of the infinite connected 
component of its complement, 
\[ K=\partial ((\C\setminus K)_{\infty}). \]
Then $K$ is thin.
\end{proposition}

Indeed, let $P$ be a polynomial mapping and $J$ be its Julia set. Then, the infinite connected component $(\C \setminus J)_{\infty}$ is the basin of attraction of the point at infinity, and $J$ is its boundary, see \cite[Lemma 17.1]{M}.

The remainder of this Section is devoted to the proof of  Proposition~\ref{p:C-K}.
We start with a geometric assertion. Consider a continuous path $\gamma_0:[1,2]\to \C$, that is a piecewise-affine path, joining the points (see Fig.~\ref{f:gamma-0})
\[
\gamma_0(1)=(-2+i),  \quad
\gamma_0(1.25)=(1+i), \quad
\gamma_0(1.5)=(1-i), 
\]
\[
\gamma_0(1.75)=(-1-i),\quad
\gamma_0(2)=(-1+2i).
\]

\begin{figure}
\includegraphics[height=5cm]{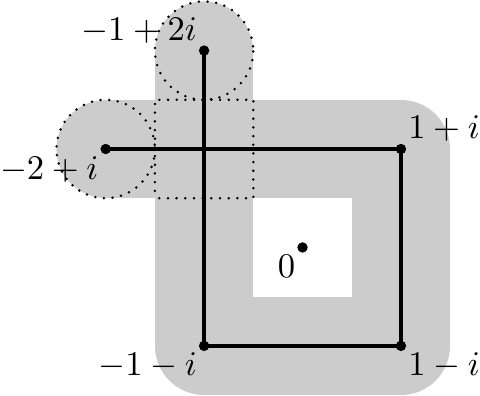}
\caption{Path $\gamma_0$ and its $\frac{1}{2}$-neighbourhood}\label{f:gamma-0}
\end{figure}

\begin{lemma}\label{l:1}
  Any continuous path $\gamma : [1,2]\to \C$ such that $\| \gamma - \gamma_0\|_{C([1,2])}< 1/2$ separates $0$ from $\infty$, that is, $0$ belongs to a bounded connected component of the complement $\C\setminus \gamma([1,2])$.
\end{lemma}

\begin{proof}
See Fig.~\ref{f:gamma-0}; note that parts of the path $\gamma |_{[1,1.25]}$ and $\gamma |_{[1.75,2]}$ have an intersection point inside the dotted square.
\end{proof}

Now recall that any open subset in the Banach space $C([1,2])$ has positive Wiener measure,  see e.g. \cite[Exercise~1.8]{MP}.
For the remainder of the proof, we fix $p_0>0$ such that $\|\gamma-\gamma_0\|_{C([1,2])}<\frac{1}{3}$ holds with probability at least $p_0$ for Brownian paths $\gamma : \R^+ \to \C$ starting at $0$ in $\C$.
The next proposition asserts that conditioning a Brownian path to reach a given point at some large moment of time affects its renormalized behaviour near the starting moment smaller and smaller. 

\begin{proposition}\label{p:gamma-close}
Let $b \in \C$ and $T > 0$. Then, conditionally to $\gamma(T)=b$ (or to any conditioning of \(\gamma\) on  $[T,+\infty)$), there exists $$T' = F(b,T) < T/2$$ such that the path $$\gamma'(t):= \frac{1}{\sqrt{T'}} \gamma(T' t)$$ (that can be seen as a renormalized restriction of $\gamma$ on $[T',2T']$) satisfies 
\begin{equation}\label{eq:close}
\|\gamma'-\gamma_0\|_{C([1,2])}<\frac{1}{2}
\end{equation}
with probability at least $p_0 / 2$.
\end{proposition}

\begin{proof}
The law of $\gamma|_{[0,T]}$ conditionally to $\gamma(T)=b$ is the same as the law of 
\[
B_t+ \frac{t}{T} (b-B_T), \quad t\in [0,T],
\]
where $B_t$ is a Brownian motion. This implies that $\gamma'|_{[1,2]}$ is distributed as
\begin{equation}\label{eq:T'}
\frac{1}{\sqrt{T'}} B_{T't} + \frac{\sqrt{T'} t}{T} (b-B_T), \quad t\in [1,2].
\end{equation}
Let $C > 0$ such that $| b-B_T | <C \cdot T$ holds with probability at least $1- p_0 / 2$. Now, fixing $T'$ small enough such that $\sqrt{T'} \cdot 2C < 1/6$, 
we ensure that with probability at least $1-p_0 / 2$, the second summand in~\eqref{eq:T'} does not exceed $1/6$ for any $t\in[1,2]$. Meanwhile, the first summand is again a Brownian motion, so it is $1/3$-close to $\gamma_0$ with probability at least~$p_0$ due to Lemma~\ref{l:1}. 

Hence, with the probability at least $p_0- p_0 / 2=p_0 / 2$ one has 
$$ \|\gamma'-\gamma_0\|_{C([1,2])} \le \|\gamma_0 -\frac{1}{\sqrt{T'}} B_{T't} \|_{C([1,2])}  + \| \frac{\sqrt{T'} t}{T} (b-B_T) \|_{C([1,2])}  , $$
which is smaller than $1/3 + 1/6 = 1/2$.
\end{proof}

The event described in Proposition \ref{p:gamma-close} implies that the path has left $K$:
\begin{lemma}\label{l:leave}
Let $T'$ be given, and assume that the path
\[
\gamma'(t):= \frac{1}{\sqrt{T'}} \gamma(T' t)
\]
satisfies~\eqref{eq:close}. Then for any $x_0\in K$ the path $x_0+ \gamma(t),$ $t\in [T',2T']$ cannot be contained in~$K$.
\end{lemma}
\begin{proof}
If one had $x_0+ \gamma([T',2T']) \subset K$, then due to Lemma~\ref{l:1} there would be a neighborhood of~$x_0$ consisting of points that one cannot connect to infinity without crossing $K$.
And this would contradict the assumption that arbitrarily close to $x_0$ there are points in the unbounded connected complement~$(\C\setminus K)_{\infty}$. 
\end{proof}

We conclude the proof of Proposition~\ref{p:C-K} by iteratively looking at the Brownian path closer and closer to $t=0$. Namely, let us show that for any $x_0\in K$ and for arbitrarily small $\delta>0$ almost surely the path $x_0+\gamma(t)$ is not contained in~$K$ (where $\gamma(t)$ is the standard Brownian path, starting at $0$). Indeed, let us construct a sequence of random times, defined by 
\[
T_0:=\delta, \quad T_n=F(\gamma(T_{n-1}),T_{n-1}), \quad n=1,2,\dots.
\]
Observe that $T_{n+1} < {1 \over 2}T_n$. Let $\gamma_n(t):= \frac{1}{\sqrt{T_n}} \gamma(T_n t)$ for $t\in [1,2]$, and let
\[
A_n := \{ \text{\eqref{eq:close} holds for } \gamma'=\gamma_n\}.
\]
Consider also the $\sigma$-algebrae $\FF_n$, generated by $T_{n}$ and $\gamma|_{[T_{n},\infty)}$. Then, on one hand, the event $A_n$ is $\FF_n$-measurable. On the other hand, due to the choice of the function~$F$ and to Proposition~\ref{p:gamma-close} conditionally to any event in $\FF_n$, the probability of $A_{n+1}$ is at least $p_0/2$. Hence, for the events 
\[
B_n:=\{\forall i=1,2,\dots, n \quad A_i \text{ does not hold}\}
\]
one has
\[
\Prob( B_{n+1} ) \le (1-p_0/2) \cdot \Prob( B_{n+1} ),
\]
and thus 
\[
\Prob(B_n) \le (1-p_0/2)^n.
\]
Thus, almost surely, at least one of the events $A_n$ takes place. By Lemma~\ref{l:leave}, this implies that $x_0+\gamma([T_n,2T_n])$ is not contained in $K$, and hence the path $x_0+\gamma(t)$ leaves~$K$ no later than $2T_n<T_0=\delta$. That completes the proof of Proposition ~\ref{p:C-K}. \\


\end{document}